\documentclass[12pt]{amsart}
\usepackage{
	amsmath,  amssymb,  amsthm,   amscd,
	gensymb,  graphicx, etoolbox, booktabs,
	stackrel, mathtools    
}
\usepackage[usenames,dvipsnames]{xcolor}
\definecolor{darkblue}{rgb}{0,0,0.7}
\definecolor{darkred}{rgb}{0.7,0,0}
\usepackage[colorlinks=true, linkcolor=darkred, citecolor=darkblue, urlcolor=blue, pagebackref=true, breaklinks=true]{hyperref}
\usepackage[capitalise]{cleveref}
\usepackage{placeins}
\usepackage{relsize}
\usepackage{todonotes}
\usepackage{soul}
\usepackage{enumitem}
\usepackage{setspace}
\usepackage{wrapfig}
\usepackage{etoolbox}
\BeforeBeginEnvironment{wrapfigure}{\setlength{\intextsep}{0pt}}
\usepackage{tikz}
\usetikzlibrary{arrows}
\usetikzlibrary{shapes}
\usepackage{float}
\usepackage{tabularx}
\usepackage{kantlipsum}
\usepackage{array}

\newtheorem{proposition}{Proposition}[section]
\newtheorem{lemma}[proposition]{Lemma}
\newtheorem{theorem}[proposition]{Theorem}

\newtheorem{corollary}[proposition]{Corollary}

\theoremstyle{definition}
\newtheorem{remark}[proposition]{Remark}

\newtheorem{definition}[proposition]{Definition}

\newtheorem{innercustomthm}{Theorem}
\newenvironment{customthm}[1]
{\renewcommand\theinnercustomthm{#1}\innercustomthm\itshape}
{\endinnercustomthm}

\newtheorem{innercustomcor}{Corollary}

\advance\headheight1.15pt

\numberwithin{equation}{section}
\def\K{\mathbb{K}}
\def\G{\mathcal{G}}
\def\C{\mathcal{C}}
\def\B{\mathcal{B}}
\def\P{\mathcal{P}}
\def\E{\mathcal{E}}
\def\N{\mathcal{N}}
\def\D{\mathcal{D}}
\def\Gp{G^{\pi}}
\def\s{\setminus}
\def\d{\mathrm{depth}}
\def\reg{\mathrm{reg}}
\def\lex{\mathrm{lex}}
\def\lcm{\mathrm{lcm}}
\def\set{\mathrm{set}}
\def\hs{\mathrm{HS}}
\def\aim{\mathrm{aim}}
\def\l{\langle}
\def\r{\rangle}
\def\x{\mathbf x}
\def\b{\mathbf b}
\def\Bl{\mathcal{B}_G}
\def\height{\mathrm{ht}}
\def\H{\mathcal{H}}

\begin{document}

\title[Homological shifts of cover ideals]{On the homological shifts of cover ideals of Cohen-Macaulay graphs}

\author{Amit Roy}
\address{Chennai Mathematical Institute, India}
\email{amitiisermohali493@gmail.com}

\author{Kamalesh Saha}
\address{Chennai Mathematical Institute, India}
\email{kamalesh.saha44@gmail.com; ksaha@cmi.ac.in}


\vspace*{-0.4cm}
\begin{abstract}
For a non-negative integer $k$, let $\hs_{k}(J(G))$ denote the $k^{\text{th}}$ homological shift ideal of the vertex cover ideal $J(G)$ of a graph $G$. For each $k\ge 2$, we construct a Cohen-Macaulay very well-covered graph $G_k$ which is both Cohen-Macaulay bipartite and a whiskered graph so that $\hs_{k}(J(G_k))$ does not have a linear resolution. This contradicts several results as well as disproves a conjecture in [J. Algebra, $\mathbf{629}$, (2023), 76-108] and [Mediterr. J. Math., $\mathbf{21}$, 135 (2024)]. The graphs $G_k$ are also examples of clique-whiskered graphs introduced by Cook and Nagel, which include Cohen-Macaulay chordal graphs, Cohen-Macaulay Cameron-Walker graphs, and clique corona graphs. Surprisingly, for Cohen-Macaulay chordal graphs, we can use a special ordering on the minimal generators to show that $\hs_k(J(G))$ has linear quotients for all $k$. Moreover, for all Cohen-Macaulay Cameron-Walker graphs and certain clique corona graphs, we show that $\hs_{k}(J(G))$ is weakly polymatroidal, and thus, has linear quotients for all $k$.
\end{abstract}

\maketitle

\section{Introduction}

The $\mathbb N^n$-graded minimal free resolution of a monomial ideal provides a powerful tool for understanding the intrinsic algebraic structure of the ideal. In particular, the syzygy modules arising in the resolution play a central role in understanding the fine properties of the ideal. The study of these syzygies and the structure of the associated graded free resolution remains an active and ongoing area of research in commutative algebra. Let $\K$ be a field, and $I$ a monomial ideal in the polynomial ring $R=\K[x_1,\ldots,x_n]$. Then the $\mathbb N^n$-graded minimal free resolution of $I$ is an exact sequence of the form 
	\[
	\mathcal F_{\cdot}: \,\, 0\rightarrow F_r\xrightarrow{\partial_{r}}\cdots\xrightarrow{\partial_{2}} F_1\xrightarrow{\partial_1} F_0\xrightarrow{\partial_0} I\rightarrow 0, 
	\]
	where $F_k=\oplus_{\mathbf a_{kj}\in\mathbb N^n}R(-\mathbf a_{kj})$ for each $k\ge 0$, and $R(-\mathbf a_{kj})$ is the polynomial ring $R$ with its grading twisted by the integer vector $\mathbf a_{kj}$. The vectors $\mathbf a_{kj}$ are called the $k^{\text{th}}$ multigraded shifts of $I$, and the number $r$ is called the projective dimension of $I$, denoted by $\mathrm{pd}(I)$.

    To gain a general understanding of the multigraded shifts in a graded minimal free resolution, a recent line of research emerged with the introduction of homological shift ideals. These ideals first appeared implicitly in the book by Miller and Sturmfels \cite[Theorem 2.18]{MillerStrumfelsBook} in 2005. They later surfaced explicitly as ideals generated by multigraded shifts in the works of Bayati et al. \cite{BayatiMatroidal, BayatietcLinQuoMultShiftBorel}. However, the concept drew significant attention following the influential work of Herzog et al. \cite{HMRZ2021} in 2020, where the term `homological shift ideal' was formally introduced.
    
     Let $k$ be a non-negative integer. The $k^{\text{th}}$ {\it homological shift ideal} of $I$, denoted by $\hs_k(I)$, is the monomial ideal generated by the monomials $\x^{\mathbf a_{kj}}$, where $\mathbf a_{kj}$ is a $k^{\text{th}}$ multigraded shift of $I$. Observe that $\hs_0(I)=I$ and $\hs_k(I)\neq\l 0\r$ if and only if $k\le \mathrm{pd}(I)$. The study of homological shift ideals has been carried out for various classes of monomial ideals, including polymatroidal ideals \cite{BandariPolyLinRes, BayatiMatroidal, BayatiQuasiAdditiveHomShift, FicarraHomShiftPoly}, Borel type ideals \cite{BayatietcLinQuoMultShiftBorel, HerzogEtcBorelType}, edge ideals of graphs \cite{KanoyTrungAryamanHomShiftCochordal, FicarraHerzogDiracMultiSyz, FicarraQureshiHomShiftAlg, FicarraQureshiEdgeAsymSyz, BayatietcHomLinQuoEdgeIdealGraph}, and the vertex cover ideals \cite{CrupiFicarraVeryWellCoveredJAlg, CrupiFicarraVeryWellCoveredMedd}, among others. It is important to note that certain multigraded shifts appearing in the minimal free resolution of $I$ may not be detected by the minimal generators of the homological shift ideals $\hs_k(I)$. However, this phenomenon does not arise if the ideal in consideration admits a linear resolution or satisfies the linear quotient property. Thus, monomial ideals with linear resolutions or the linear quotient property hold particular significance in the study of homological shift ideals.
	
	In this article, we focus on the homological shift ideals of the vertex cover ideal of a graph. Let $G$ be a finite simple graph on the vertex set $V_G=\{x_1,\ldots,x_n\}$ with edge set $E_G$. By identifying the vertices with the indeterminates, we work in the polynomial ring $R$. A subset $\C\subseteq V_G$ is called a {\it vertex cover} of $G$ if $\C\cap e\neq\emptyset$ for each $e\in E_G$. A {\it minimal vertex cover} is a vertex cover that is minimal with respect to inclusion. Then the ideal
	\[
	J(G)=\l \x_{\C}\mid\C\text{ is a minimal vertex cover of }G \r
	\]
	is called the {\it vertex cover ideal} of $G$, where for any subset $A\subseteq\{x_1,\ldots,x_n\}$, $\x_{A}$ denotes the monomial $\prod_{x_i\in A}x_i$ in $R$. Recall that the vertex cover ideal $J(G)$ is the Alexander dual of the edge ideal $I(G)$ of $G$. By the well-known Eagon-Reiner \cite{EagonReiner1998} theorem, $J(G)$ has linear resolution if and only if $I(G)$ is Cohen-Macaulay, and in that case we say that $G$ is a Cohen-Macaulay graph. In this paper, we restrict ourselves to the vertex cover ideal of a Cohen-Macaulay graph. 
	
	An important question in the study of homological shift ideals is to investigate which properties of the monomial ideal $I$ are preserved by all its homological shift ideals $\hs_{k}(I)$. For instance, is it true that if $I$ has a linear resolution (respectively, linear quotients), then so does $\hs_{k}(I)$, for all $k\ge 0$? In general, the answers to both questions are negative; see, for instance, \cite[Example 1.4]{FicarraHerzogDiracMultiSyz} and \cite[Remark 2.2]{BayatietcHomLinQuoEdgeIdealGraph}. Nevertheless, it remains an interesting problem to identify classes of monomial ideals for which the linear resolution or the linear quotient property extends to all of their homological shift ideals. As mentioned before, the vertex cover ideal of a Cohen-Macaulay graph has linear resolution. Note that, in general, all monomial ideals having linear quotient property also have linear resolution, but the converse is not true \cite{HHBook}. For many classes of Cohen-Macaulay graphs, namely, the Cohen-Macaulay very well-covered graphs \cite{KimuraPournakiFakhariTeraiYassemi}, Cohen-Macaulay chordal graphs \cite{HHBook}, Cohen-Macaulay Cameron-Walker graphs \cite{HibiHigashitaniKimuraCamWalk}, etc., their vertex cover ideals also have the linear quotient property.
	
	The homological shift ideals of the vertex cover ideals of graphs are studied by Crupi and Ficarra \cite{CrupiFicarraVeryWellCoveredJAlg, CrupiFicarraVeryWellCoveredMedd}. In particular, it was shown in \cite[Theorem 4.1]{CrupiFicarraVeryWellCoveredJAlg} that if $G$ is a Cohen-Macaulay very well covered graph, then $\hs_{k}(J(G))$ admits linear quotients for all $k\ge 0$. In fact, as shown in \cite[Theorem 4.2]{CrupiFicarraVeryWellCoveredJAlg}, a very well-covered is Cohen-Macaulay if and only if all homological shift ideals of the Alexander dual of its cover ideals have linear quotients. Furthermore, it was conjectured in \cite[Conjecture 4.4]{CrupiFicarraVeryWellCoveredJAlg}  that $\hs_{k}(J(G)^l)$ has linear quotients for all $k\ge 0$. The conjecture has been verified for the class of Cohen-Macaulay bipartite graphs \cite[Corollary 4.11]{CrupiFicarraVeryWellCoveredJAlg}, and whiskered graphs \cite[Theorem 4.8]{CrupiFicarraVeryWellCoveredMedd}. However, our first main result in this direction demonstrates that for each $k\ge 2$, there exists a Cohen-Macaulay very well-covered graph $G_k$ which is also a Cohen-Macaulay bipartite graph as well as a whiskered graph so that the $k^{th}$ homological shift ideal of its vertex cover ideal fails to satisfy the linear quotient property. In fact, we establish the following stronger result:
	    \begin{customthm}{\ref{theorem 200}}
	    	For each $k\ge 2$, there exists a Cohen-Macaulay bipartite graph $G_k$, which is also a whiskered graph such that the ideal $\hs_{k}(J(G_k))$ does not have linear resolution and hence, $\hs_{k}(J(G_k))$ does not satisfy the linear quotient property.
	    \end{customthm}

        \noindent
	    One can see that \Cref{theorem 200} implies \cite[Theorem 4.1, Theorem 4.2, Corollary 4.11]{CrupiFicarraVeryWellCoveredJAlg}, \cite[Theorem 4.8]{CrupiFicarraVeryWellCoveredMedd} as well as \cite[Conjecture 4.4]{CrupiFicarraVeryWellCoveredJAlg} are not true.\par 
        
        Incidentally, the graphs $G_k$ are examples of clique-whiskered graphs, introduced by Cook and Nagel in \cite{CookNagel2012}. The clique-whiskered graphs are combinatorially rich in the sense that the $f$-vector of the independence complex of any graph can also be expressed as the $h$-vector of the independence complex of some clique-whiskered graph (see \cite{CookNagel2012}, for many other interesting results in this direction). As the name suggests, clique-whiskered graphs are also a generalization of whiskered graphs, first introduced by Villarreal \cite{CM}. The minimal free resolution of the vertex cover ideal of a clique-whiskered graph has been explicitly described in a recent preprint by Muta and Terai \cite{MutaTeraiMinFreeResCliqWhis}. The clique-whiskered graphs are vertex decomposable, and in particular, Cohen-Macaulay. Moreover, they contain several well-known subclasses of Cohen-Macaulay graphs, including Cohen-Macaulay chordal graphs, Cohen-Macaulay Cameron-Walker graphs, clique corona graphs, etc. In light of \Cref{theorem 200}, a natural question arises: for which classes of clique-whiskered graphs do all homological shift ideals of their vertex cover ideals possess the linear quotient property? In this context, we first establish the following result:
	    \begin{customthm}{\ref{chordal theorem}}
	    	If $G$ is a Cohen-Macaulay chordal graph, then for each $k\ge 0$, $\hs_k(J(G))$ has linear quotients.
	    \end{customthm}
	    
	    \noindent
	    To prove \Cref{chordal theorem}, we employ the technique of Betti splitting of vertex cover ideals of chordal graphs and construct a carefully designed inductive ordering on the minimal generators of the corresponding homological shift ideals. Here we remark that this type of ordering on minimal generators might be useful in the future study of homological shift ideals. Moving on, our next main result concerns the homological shift ideals of Cohen-Macaulay Cameron-Walker graphs.
	    
	    \begin{customthm}{\ref{Cameron-Walker graph}}
	    	If $G$ is a Cohen-Macaulay Cameron-Walker graph, then for each $k\ge 0$, $\hs_k(J(G))$ is weakly polymatroidal.
	    \end{customthm}
	    The notion of weakly polymatroidal ideal generated in a fixed degree was introduced by Kokubo and Hibi \cite{KokubuHibiWeakPoly}, and later extended to arbitrary monomial ideals by Mohammadi and Moradi \cite{MohammadiMordaliWeakPoly}. In \cite[Theorem 1.3]{MohammadiMordaliWeakPoly} it was shown that all weakly polymatroidal ideals admit linear quotients. Consequently, all homological shift ideals of the vertex cover ideal of a Cohen-Macaulay Cameron-Walker graph possess the linear quotient property.
	    
	    The concept of clique corona graphs was introduced by Hoang and
	    Pham in \cite{HoangPhamCliqueCorona}, where they proved that the clique corona graphs are the only class of corona graphs that are Cohen-Macaulay. Notably, the graphs $G_k$ in \Cref{theorem 200} are examples of clique-corona graphs. Therefore, in general, it is not the case that all homological shift ideals of the vertex cover ideal of clique corona graphs satisfy the linear quotient property. However, in \Cref{clique corona} we have shown that for certain classes of clique corona graphs, the linear quotient property does hold. In fact, for these graphs, the corresponding homological shift ideals are weakly polymatroidal.

\section{Main results}

    This section presents all of our main results. We begin by examining the homological shift ideals of the vertex cover ideals of the graphs $G_k$. As noted earlier, the graphs $G_k$, along with Cohen-Macaulay chordal graphs, Cohen-Macaulay Cameron-Walker graphs, and clique corona graphs, are all examples of clique-whiskered graphs introduced by Cook and Nagel \cite{CookNagel2012}. Therefore, our first goal in this section is to analyze the homological shift ideals of the vertex cover ideals of clique-whiskered graphs. Before proceeding, let us fix some notations and terminologies that will be used throughout the paper.

    Let $G$ be a graph on the vertex set $V_G$ and the edge set $E_G$. For any $W\subseteq V_G$, $G\setminus W$ denotes the graph on the vertex set $V_G\setminus W$ with edge set $\{e\in E_G\colon e\cap W=\emptyset\}$. In this case, we call $G\setminus W$ an {\it induced subgraph} of $G$ on the vertex set $V(G)\setminus W$. If $W=\{a\}$ for some $a\in V_G$, then $G\setminus\{a\}$ is simply denoted by $G\setminus a$. For a vertex $a\in V_G$, the set $\{b\in V_G\mid \{a,b\}\in E_G\}$ is called the (open) {\it neighborhood} of $G$, and is denoted by $\N_G(a)$. The {\it closed neighborhood} of $a$ is the set $\N_G[a]:=\N_G(a)\cup\{a\}$. The {\it cycle} graph $C_n$ on the vertex set $\{x_1,\ldots,x_n\}$ consists of the edges $\{\{x_i,x_{i+1}\},\{x_1,x_n\}\colon i\in[n-1]\}$. Here for any positive integer $k$, the set $\{1,\ldots,k\}$ is denoted in short by $[k]$. A graph on $n$ vertices is said to be \textit{complete} if there is an edge between every pair of vertices and is denoted by $K_n$.

    Let us recall the construction of clique-whiskered graphs from \cite{CookNagel2012}. A subset $A$ of the vertex set of a graph $G$ is said to be a {\it clique} of $G$ if it induces a complete subgraph of $G$. A {\it clique vertex partition} of $G$ is a set $\pi = \{A_1 ,\dots, A_t \}$ of disjoint cliques of $G$ such that their union forms $V_G$. For each $i\in[t]$, let $A_i=\{w_{i1},\ldots,w_{ir_i}\}$. The {\it clique-whiskered graph} $\Gp$ for the clique vertex partition $\pi$ is  defined as follows:
    \begin{align*}
        V_{\Gp}&=V_G\sqcup\{v_1,\ldots,v_t\},\\
        E_{\Gp}&=E_G\sqcup(\cup_{i=1}^t\{\{v_i,w_{ij}\}\colon j\in [r_i]\}).
    \end{align*}
 The following two propositions are helpful to find the minimal generators of $\hs_k(J(\Gp))$.
	
	\begin{proposition}\label{proposition 4}
		Let $\C$ be a minimal vertex cover of $\Gp$. Then for each $i\in[t]$, we have $|\N_{\Gp}[v_i]\cap\C|=|\N_{\Gp}[v_i]|-1$.
	\end{proposition}
	\begin{proof}
		Since $\N_{\Gp}[v_i]$ forms a complete graph, we have that $|\N_{\Gp}[v_i]\cap\C|\ge|\N_{\Gp}[v_i]|-1$. On the other hand, if $\N_{\Gp}[v_i]\cap\C=\N_{\Gp}[v_i]$, then $\C\s\{v_i\}\subsetneq\C$ is a vertex cover of $\Gp$, a contradiction.	
	\end{proof}
	
	\begin{proposition}\label{proposition 2}
		Let $\C$ be a minimal vertex cover of $\Gp\s w_{ij}$ for some $i\in[t]$ and $j\in [r_i]$. Then $\C\s \N_{\Gp}(w_{ij})$ is a vertex cover of $\Gp\s\N_{\Gp}[w_{ij}]$.
	\end{proposition}
    \begin{proof}
        Since $\Gp\s\N_{\Gp}[w_{ij}]$ is an induced subgraph of $\Gp\s w_{ij}$, the result follows.
    \end{proof}

    \begin{definition}[linear quotient]
        Let $I$ be a monomial ideal with minimal generating set $\G(I)$. We say that $I$ has linear quotients if there exists a total order on $\G(I)$, say $m_{1}>m_{2}>\dots>m_{r}$ such that the ideals
        \[
\l m_{1},\ldots,m_{j}\r:m_{j+1}
        \]
        are generated by variables for each $j\in[r-1]$.
    \end{definition}

To find the minimal generators of $\hs_k(J(\Gp))$, we define a total order on $V_{\Gp}$ as follows:
\begin{align}\label{eq 200}
    w_{11}>w_{12}>\dots>w_{1r_1}>v_1>w_{21}>\dots>v_{t-1}>w_{t1}>w_{t2}>\dots>w_{tr_t}>v_t.
\end{align}
The fact that $\Gp$ is a Cohen-Macaulay graph implies the vertex cover ideal $J(\Gp)$ is an equigenerated ideal. In general, if $I$ is any equigenerated monomial ideal having linear quotients with respect to a total order on the minimal generators, say  $m_1>m_2>\cdots>m_r$, then one can define, for each $i\in\{2,\ldots,r\}$, $\set_{I}(m_i)=\{x_{i_1},\ldots,x_{i_s}\}$, where $\l m_1,\ldots,m_{i-1}\r:m_i=\l x_{i_1},\ldots,x_{i_s}\r$. In the special case, when $H$ is a graph with the vertex cover ideal having linear quotients, and $\C$ is a minimal vertex cover of $H$ such that $m=\x_{\C}$ is a minimal generator of $J(H)$, we write $\set_{J(H)}(m)$ simply as $\set_{H}(\C)$.

In the next theorem, we show that $J(\Gp)$ has linear quotients with respect to the lexicographic order induced from the total order of $V_{\Gp}$ in \Cref{eq 200}. Moreover, we explicitly determine $\set_{\Gp}(\C)$, where $\C$ is a minimal vertex cover of $\Gp$. Here we remark that the fact $J(\Gp)$ has linear quotients is well-known (cf. \cite{CookNagel2012}). However, the combinatorial description of $\set_{\Gp}(\C)$ in \Cref{theorem 1} is new, and will play a crucial role in the subsequent part of this section.

	\begin{theorem}\label{theorem 1}
		$J(\Gp)$ has linear quotients with respect to the lexicographic order induced from the total order $>$ on ${V_{\Gp}}$ in \Cref{eq 200}. Moreover, for each $\x_{\C}\in\G(J(\Gp))$ we have $\mathrm{set}_{\Gp}(\C)=\cup_{i=1}^t( \N_{\Gp}(v_i)\setminus \C)$.
	\end{theorem}
	\begin{proof}
		Let $\C$ and $\C'$ be two minimal vertex cover of $G$ such that $\x_{\C'}>\x_{\C}$ and let 
		\[
		h=\frac{\mathrm{lcm}(\x_{\C'},\x_{\C})}{\x_{\C}}.
		\]
		If $\deg(h)=1$, then there is nothing to prove. Therefore, we may assume that $\deg h\ge 2$. Let $z=\max\{a\in V(\Gp)\mid a\in\C'\setminus \C\}$. Since $\x_{\C'}>\x_{\C}$, we observe that for each $b>z$ we have $b\in\C'$ if and only if $b\in\C$. Now suppose $z\in \N_{\Gp}[v_i]$ for some $i\in [t]$. If $z=v_i$, then by \Cref{proposition 4}, $w_{ij}\notin\C'$ for some $j\in[r_i]$. On the other hand, $z=v_i\notin\C$ implies $w_{ij}\in\C$, which is a contradiction to the fact that $\x_{\C'}>\x_{\C}$. Thus we must have $z\in \N_{\Gp}(v_i)$. Now suppose $z=w_{il}$ for some $l\in[r_i]$. Then $w_{il}\mid h$. Moreover, if $\C''=(\C\setminus\{v_i\})\cup\{w_{il}\}$, then using \Cref{proposition 4} one can see that $\C''$ is a minimal vertex cover of $\Gp$ so that $\x_{\C''}>\x_{\C}$, where
		\[
		w_{il}=\frac{\mathrm{lcm}(\x_{\C''},\x_{\C})}{\x_{\C}}.
		\]
		Thus $J(\Gp)$ has linear quotients with respect to the given monomial order.
		
		For the second part, let $w_{ij}\in \N_{\Gp}(v_i)\s\C$ for some $i\in [t]$ and $j\in[r_i]$. Then it is easy to see that $\x_{\C_1}>\x_{\C}$, where $\C_1=(\C\setminus\{v_i\})\cup\{w_{ij}\}$ is a minimal vertex cover of $G$. Moreover, 
		\[
		w_{ij}=\frac{\mathrm{lcm}(\x_{\C_1},\x_{\C})}{\x_{\C}}.
		\]
		Thus we have $\cup_{i=1}^t( \N_{\Gp}(v_i)\setminus \C)\subseteq\mathrm{set}_{\Gp}(\C)$. To prove the other inclusion, it is enough to show that $v_i\notin\mathrm{set}_{\Gp}(\C)$ for each $i\in[t]$ since $\C\cap\mathrm{set}_{\Gp}(\C)=\emptyset$. Let $\C'$ be a minimal vertex cover of $\Gp$ such that $\x_{\C'}>\x_{\C}$ and $v_i\in \C'\setminus \C$. As before, let 
		\[
		h=\frac{\mathrm{lcm}(\x_{\C'},\x_{\C})}{\x_{\C}}.
		\]
		Then $v_i\mid h$. Let $v_l=\max\{v_i\mid v_i\in\C'\setminus \C\}$. Then for each $w_{lj}\in \N_{\Gp}(v_l)$, we have $w_{lj}\in\C$, whereas there exists some $w_{lm}\in \N_{\Gp}(v_l)$ such that $w_{lm}\notin\C'$. Therefore, since $\x_{\C'}>\x_{\C}$, we have some $v_p>v_l$ and $w_{pn}\in \N_{\Gp}(v_p)$ such that $w_{pn}\in\C'\s\C$. Thus $w_{pn}\mid h$ and hence $v_i\notin\mathrm{set}_{\Gp}(\C)$, as desired.
	\end{proof}

    From \Cref{theorem 1} one can easily verify the following.
	\begin{corollary}\label{proposition 3}
		Let $\C$ be a minimal vertex cover of $\Gp$ and $\sigma\subseteq V_{G}$ such that $\sigma\cap\C=\emptyset$. Then $\sigma\subseteq\mathrm{set}_{\Gp}(\C)$.
	\end{corollary}
	
	\begin{remark}\label{remark1}
		It follows from \cite[Page 4]{HMRZ2021} that for each $k\ge 0$, the minimal monomial generators of $\hs_k(J(\Gp))$ are of the form $\x_{\C}\x_{\sigma}$, where $\C$ is a minimal vertex cover of $\Gp$, and $\sigma\subseteq\set_{\Gp}(\C)$ with $|\sigma|=k$. Thus by \Cref{theorem 1} and \Cref{proposition 3}, we can say $\sigma\subseteq V_G$ with $\C\cap\sigma=\emptyset$ and $|\sigma|=k$. 
	\end{remark}

    \noindent Building on the above observations, we now proceed to construct the graphs $G_k$ as follows. For each integer $k\ge 2$, let $G_k$ denote the graph described as
	\begin{align*}
		V_{G_k}&=\{x_i,y_i\colon i\in[2k]\},\\
		E_{G_k}&=\{\{x_i,y_i\},\{x_i,x_{i+1}\}\colon i\in[2k]\}.
	\end{align*}
   \noindent Here we assume that $x_{2k+1}=x_{1}$. For a pictorial view of the graph $G_k$, see \Cref{Gk graph}.

        \begin{figure}[h!]\
		\centering
		\begin{tikzpicture}
			[scale=.55]
			\draw [fill] (4,8) circle [radius=0.1];
			\draw [fill] (4,6) circle [radius=0.1];
			\draw [fill] (6,8) circle [radius=0.1];
			\draw [fill] (6,6) circle [radius=0.1];
			\draw [fill] (0,4) circle [radius=0.1];
			\draw [fill] (2,4) circle [radius=0.1];
			\draw [fill] (8,4) circle [radius=0.1];
			\draw [fill] (10,4) circle [radius=0.1];
			\draw [fill] (4,2) circle [radius=0.1];
			\draw [fill] (6,2) circle [radius=0.1];
			\draw [fill] (4,0) circle [radius=0.1];
			\draw [fill] (6,0) circle [radius=0.1];
			\node at (4,8.7){$y_1$};
			\node at (3.2,6){$x_1$};
			\node at (6.8,6){$x_2$};
			\node at (6,8.7){$y_2$};
			\node at (-0.5,4.7){$y_{2k}$};
			\node at (1.8,4.7){$x_{2k}$};
			\node at (8.2,4.6){$x_3$};
			\node at (10.2,4.6){$y_3$};
			\node at (2.7,2){$x_{2k-1}$};
			\node at (5.05,2){$\cdots$};
			\node at (6.8,2){$x_4$};
			\node at (2.7,0){$y_{2k-1}$};
			\node at (6.8,0){$y_4$};
			\draw (4,0)--(4,2)--(2,4)--(4,6)--(6,6)--(8,4)--(6,2)--(6,0);
			\draw (2,4)--(0,4);
			\draw (4,6)--(4,8);
			\draw (6,6)--(6,8);
			\draw (8,4)--(10,4);
			\draw (6,2)--(6,0);
		\end{tikzpicture}
		\caption{The graph $G_k=W(C_{2k})$.}\label{Gk graph}
	\end{figure}
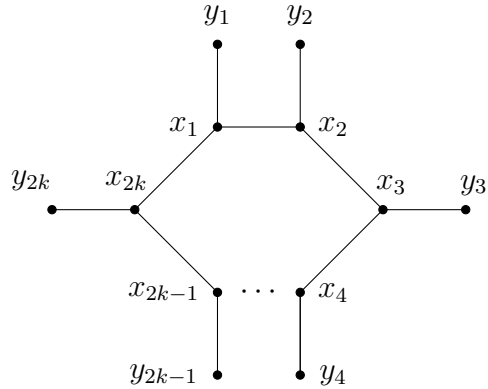

	It is easy to see that $G_k$ is an example of a clique-whiskered graph. In fact, $G_k$ is the whiskered graph on the cycle of length $2k$. Indeed, if $C_{2k}$ denotes the cycle of length $2k$ on the vertex set $\{x_1,\ldots,x_{2k}\}$ with clique vertex-partition $\pi=\{\{x_1\},\ldots,\{x_{2k}\}\}$, then $G_{k}=C_{2k}^{\pi}$. Using the description of the minimal generators of $\hs_{k}(J(\Gp))$ from \Cref{theorem 1} and \Cref{remark1}, we prove the following.
	
	\begin{theorem}\label{theorem 200}
		For each $k\ge 2$, the ideal $\hs_{k}(J(G_k))$ does not have linear resolution and hence $\hs_{k}(J(G_k))$ does not have linear quotients.
	\end{theorem} 
	\begin{proof}
		By \Cref{remark1}, $\hs_{k}(J(G_k))$ is generated by the monomials $\x_{\C}\x_{\sigma}$, where $\C$ is a minimal vertex cover of $G_k$ and $\sigma\subseteq\set_{G_k}(\C)$ with $|\sigma|=k$. Note that, if $\{x_i,x_{i+1}\}\subseteq\sigma$, then $\C$ is not a vertex cover of $G_k$ since $\sigma\cap\C=\emptyset$. Thus by \Cref{theorem 1}, $\sigma=\{x_1,x_3,\ldots,x_{2k-1}\}$ or $\{x_2,x_4,\ldots,x_{2k}\}$. Consequently, $\C=\{y_1,x_2,y_3,x_4,\ldots,y_{2k-1},x_{2k}\}$ or $\C=\{x_1,y_2,x_3,y_4,\ldots,x_{2k-1},y_{2k}\}$, respectively. Hence,
		\[
		\hs_{k}(J(G_k))=\prod_{i=1}^{2k}x_i\cdot\left\l y_1y_3\cdots y_{2k-1},y_2y_4\cdots y_{2k}\right\r.
		\]
		Observe that $\hs_{k}(J(G_k))$ is generated in degree $3k$. However, by \cite[Lemma 8]{Woodroofe2014Jca}, we have $\reg(\hs_{k}(J(G_k)))=4k-1$, where $\reg$ stands for the Castelnuovo-Mumford regularity. It is well-known that if $I$ is any equigenerated ideal in degree $d$, then $I$ has a linear resolution if and only if $\reg(I)=d$ \cite{HHBook}. Thus, $\hs_{k}(J(G_k))$ does not have linear resolution for each $k\ge 2$. The last statement follows from \cite[Proposition 8.2.1]{HHBook}.
	\end{proof}

 Our next objective is to show that all homological shift ideals of Cohen-Macaulay chordal graphs have linear quotients. To establish the linear quotient property, we use the idea of Betti splitting. The notion of Betti splitting was formally introduced by Francisco-H\`a-Van Tuyl in \cite{FranciscoHaVanTuyl2009} and has since proven to be a powerful tool with numerous applications (see \cite{FranciscoHaVanTuyl2009} for the definition and basic properties). Now, for a clique-whiskered graph $\Gp$ it follows from \cite[Theorem 3.3]{CookNagel2012} and \cite[Theorem 2.3, Theorem 2.8]{MKA} that for any $w_{ij}\in V_{\Gp}$, the $w_{ij}$-partition 
\[
J(\Gp)=\x_{w_{ij}}J(\Gp\s w_{ij})+\x_{\N_{\Gp}(w_{ij})}J(\Gp\s\N_{\Gp}[w_{ij}]),
\]
is a Betti splitting. Moreover, it follows that the intersection of the ideals
\[
\x_{w_{ij}}J(\Gp\s w_{ij})\cap\x_{\N_{\Gp}(w_{ij})}J(\Gp\s\N_{\Gp}[w_{ij}])=\x_{\N_{\Gp}[w_{ij}]}J(\Gp\s\N_{\Gp}[w_{ij}]).
\]
In this situation, we have the following result due to \cite[Proposition 1.7]{CrupiFicarraVeryWellCoveredJAlg}, which is also a direct consequence of \cite[Proposition 2.1]{FranciscoHaVanTuyl2009}.
 \begin{proposition}\label{proposition 1}
		For each $w_{ij}\in V(\Gp)$, and $k\ge 1$, we have
		\begin{align*}
		\hs_k(J(\Gp))=\x_{\N_{\Gp}[w_{ij}]}\hs_{k-1}(J(\Gp\s\N_{\Gp}[w_{ij}]))+&\x_{w_{ij}}\hs_k(J(\Gp\s w_{ij}))\\
		&+\x_{\N_{\Gp}(w_{ij})}\hs_k(J(\Gp\s\N_{\Gp}[w_{ij}])).
		\end{align*}
	\end{proposition}
    \begin{remark}\label{remark 400}
    One natural way to give an ordering on $\G(\hs_k(J(\Gp)))$ to examine the linear quotient property is by using the ordering on the three ideals mentioned in the expression of \Cref{proposition 1}. However, an immediate observation is that 
    \begin{enumerate}
        \item[$\bullet$] $\G\left(\x_{\N_{\Gp}(w_{ij})}\hs_k(J(\Gp\s\N_{\Gp}[w_{ij}]))\right)\cap \G\left(\x_{\N_{\Gp}[w_{ij}]}\hs_{k-1}(J(\Gp\s\N_{\Gp}[w_{ij}]))\right)=\emptyset$,
        
        \item[$\bullet$] $\G\left(\x_{\N_{\Gp}(w_{ij})}\hs_k(J(\Gp\s\N_{\Gp}[w_{ij}]))\right)\cap\G\left(\x_{w_{ij}}\hs_k(J(\Gp\s w_{ij}))\right)=\emptyset$,

        \item[$\bullet$] $\G\left(\x_{\N_{\Gp}[w_{ij}]}\hs_{k-1}(J(\Gp\s\N_{\Gp}[w_{ij}]))\right)\cap \G\left(\x_{w_{ij}}\hs_k(J(\Gp\s w_{ij}))\right)$ {\small may be non-empty}.
    \end{enumerate}
    Hence, in light of the third observation, particular attention must be paid to the ordering of the minimal generators of $\hs_k(J(\Gp))$. In the proof for the Cohen–Macaulay chordal case (\Cref{chordal theorem}), we provide a specific ordering that explicitly helps to address a key part of the proof.        
    \end{remark}

 \noindent The following technical lemma about the minimal generators of the homological shift ideals of $J(\Gp)$ is repeatedly used in our proofs. 
    
	\begin{lemma}\label{lemma1}
		Let $\alpha\in\G(\hs_k(J(\Gp)))$ such that  $v_i\mid\alpha$ and $w_{ij}\nmid\alpha$  for some $i\in[t]$  and $j\in[r_i]$, where $k$ is any non-negative integer. Then $\alpha'=w_{ij}\frac{\alpha}{v_i}\in\G(\hs_k(J(\Gp)))$.
	\end{lemma}
	\begin{proof}
		By \Cref{remark1}, $\alpha=\x_{\C}\x_{\sigma}$, where $\C$ is a minimal vertex cover of $\Gp$, and $\sigma\subseteq\set_{\Gp}(\C)$ with $|\sigma|=k$. Since $v_i\mid\alpha$ and $w_{ij}\nmid\alpha$, by \Cref{theorem 1}, we have $v_i\in\C$ and $w_{ij}\notin\C\cup\sigma$. Our first aim is to show that $\C'=(\C\s\{v_i\})\cup\{w_{ij}\}$ is a minimal vertex cover of $\Gp$. Indeed, it is clear from \Cref{proposition 4} that $\C\cap\N_{\Gp}[v_i]=\N_{\Gp}[v_i]\s\{w_{ij}\}$ since $v_i$ is a simplicial vertex of $\Gp$. Hence, $\C'$ is a vertex cover of $\Gp$. Let $\C''\subseteq\C'$ be a minimal vertex cover of $\Gp$. We proceed to show that $\C''=\C'$. Since $\{v_i,w_{is}\}\in E(\Gp)$ for each $s\in[r_i]$, we have that $w_{is}\in\C''$ for each $s\in[r_i]$. Now suppose $w_{lm}\in\C'$ for some $l\in[t]\s\{i\}$ and $m\in[r_l]$. This implies $w_{lm}\in\C$, and hence, by \Cref{proposition 4}, either $v_l\notin\C$ or $w_{lm_1}\notin\C$ for some $m_1\in[r_l]\s\{m\}$. Consequently, either $v_l\notin\C''$ or $w_{lm_1}\notin\C''$, respectively. Since $\{v_l,w_{lm}\},\{w_{lm},w_{lm_1}\}\in E(\Gp)$ we see that $w_{lm}\in\C''$. Similarly, we can show that if $v_p\in\C'$ for some $p\in [t]\s\{i\}$, then $v_p\in\C''$. Thus $\C'=\C''$. \par 
		
	Now it remains to show that $\sigma\subseteq\set_{\Gp}(\C')$. Indeed, it is easy to see from \Cref{theorem 1} that $\set_{\Gp}(\C')=\set_{\Gp}(\C)\s\{w_{ij}\}$. Since $w_{ij}\notin\sigma$, we have $\sigma\subseteq\set_{\Gp}(\C')$, as desired. Thus $\alpha'=\x_{\C'}\x_{\sigma}\in \G(\hs_k(J(\Gp)))$, by \Cref{remark1}. 
		\end{proof}



We are now ready to show that the homological shift ideals of the cover ideal of a Cohen-Macaulay chordal graph have linear quotients. Cohen-Macaulay chordal graphs are combinatorially classified by Herzog-Hibi-Zheng \cite{CMchordal}, and from their result it follows that a chordal graph is Cohen-Macaulay if and only if it is a clique-whiskered graph.

	\begin{theorem}\label{chordal theorem}
		Let $\Gp$ be any Cohen-Macaulay chordal graph. Then for each $k\ge 0$, $\hs_k(J(\Gp))$ has linear quotients with respect to a linear order $>$ on the minimal generators of $\hs_k(J(\Gp))$ which satisfies the following property:
		\begin{align*}
			(*)\,\,\,\,\,\text{If } \alpha\in\G(\hs_k(J(\Gp))) \text{ such that }& v_i\mid\alpha \text{ and } w_{ij}\nmid\alpha \text{ for some } i\in[t] \text{ and } j\in[r_i],\\
				& \text{ then } \alpha'=w_{ij}\frac{\alpha}{v_i}\in\G(\hs_k(J(\Gp))), \text{ and } \alpha'>\alpha.
		\end{align*}
	\end{theorem}
	\begin{proof}
		The proof is by induction on $|V_{\Gp}|$. If $k=0$, then by \Cref{theorem 1}, $\hs_0(J(\Gp))$ has linear quotients with respect to the lexicographic order induced from the total order $>$ described in \Cref{eq 200} on ${V_{\Gp}}$. Note that, the total order $>$ on $V_{\Gp}$ also satisfies the condition $(*)$ as seen in \Cref{theorem 1} and \Cref{lemma1}. Therefore, we may assume $k\ge 1$. If $|V_{\Gp}|= 2$, then $J(\Gp)\neq\l 0\r$ only when $\Gp$ is connected, and in that case $k=1$. Hence, by \Cref{theorem 1} and \Cref{remark1}, $\hs_1(J(\Gp))=\l v_1w_{11}\r$, where $V(\Gp)=\{v_1,w_{11}\}$, and thus, $\hs_1(J(\Gp))$ has linear quotients. Now suppose $|V_{\Gp}|\ge 3$. The induced subgraph $G$ is also chordal. Without loss of generality, let $w_{11}$ be a simplicial vertex of $G$. Then by \Cref{proposition 1},
		\begin{align*}
			\hs_k(J(\Gp))=\x_{\N_{\Gp}[w_{11}]}\hs_{k-1}(J(\Gp\s\N_{\Gp}[w_{11}]))+&\x_{w_{11}}\hs_k(J(\Gp\s w_{11}))\\
			&+\x_{\N_{\Gp}(w_{11})}\hs_k(J(\Gp\s\N_{\Gp}[w_{11}])).
		\end{align*}
		Observe that $|V_{\Gp\s\N_{\Gp}[w_{11}]}|<|V_{\Gp}|$, and $|V_{\Gp\s w_{11}}|<|V_{\Gp}|$. Therefore, by the induction hypothesis, all the three ideals $\hs_{k-1}(J(\Gp\s\N_{\Gp}[w_{11}]))$, $\hs_k(J(\Gp\s w_{11}))$, and $\hs_k(J(\Gp\s\N_{\Gp}[w_{11}]))$ have linear quotients with respect to some linear orderings for which the minimal monomial generators of these ideals satisfy the condition $(*)$. Let the linear quotient ordering of the minimal generators of these ideals satisfying the condition $(*)$ be as follows:
		\begin{align*}
			\G(\hs_{k-1}(J(\Gp\s\N_{\Gp}[w_{11}])))&=\{f_1>\dots>f_p\},\\
			\G(\hs_k(J(\Gp\s w_{11})))&=\{g_1>\dots>g_q\},\\
			\G(\hs_k(J(\Gp\s\N_{\Gp}[w_{11}])))&=\{e_1>\dots>e_s\}.
		\end{align*}
		Let $a=\x_{\N_{\Gp}[w_{11}]}$, $b=\x_{w_{11}}$, and $c=\x_{\N_{\Gp}(w_{11})}$. Observe from \Cref{remark 400} that $\G(c\cdot\hs_k(J(\Gp\s\N_{\Gp}[w_{11}])))\cap \G(a\cdot\hs_{k-1}(J(\Gp\s\N_{\Gp}[w_{11}])))=\emptyset$ and $\G(c\cdot\hs_k(J(\Gp\s\N_{\Gp}[w_{11}])))\cap\G(b\cdot \hs_k(J(\Gp\s w_{11})))=\emptyset$. On the other hand, the set $\G(a\cdot\hs_k(J(\Gp\s\N_{\Gp}[w_{11}])))\cap \G(b\cdot\hs_k(J(\Gp\s w_{11})))$ may be non-empty. Let
		\[
		\G(b\cdot\hs_k(J(\Gp\s w_{11})))\s\G(a\cdot\hs_k(J(\Gp\s\N_{\Gp}[w_{11}])))=\{g_{\theta_1}>\cdots>g_{\theta_m}\},
		\]  
        for some $\{\theta_1,\ldots,\theta_m\}\subseteq [q]$. Then we define an ordering on $\G(\hs_k(J(\Gp)))$ as follows:
		\begin{align}\label{eq 1}
		af_1>\cdots>af_p>bg_{\theta_1}>\cdots>bg_{\theta_m}>ce_1>\cdots>ce_s.
		\end{align}
		
		Our first aim is to show that the condition $(*)$ is satisfied for the above ordering. Observe that if $\alpha\in\G(\hs_k(J(\Gp)))$ such that  $v_i\mid\alpha$ and $w_{ij}\nmid\alpha$  for some $i\in[t]$  and $j\in[r_i]$, then by \Cref{lemma1}, we have $\alpha'=w_{ij}\frac{\alpha}{v_i}\in\G(\hs_k(J(\Gp)))$. Thus, it is enough to show that $\alpha'>\alpha$. First, consider the case when $v_i= v_1$. Then $\alpha=bg_{\theta_r}$ or $\alpha=ce_n$, where $r\in[m]$ and $n\in[s]$. In the case $\alpha=bg_{\theta_r}$, we have, by \Cref{lemma1} that $\alpha'=bg_{\theta_{r_1}}$ for some $r_1\in[m]\s\{r\}$ and thus by the induction hypothesis, $\alpha'>\alpha$. On the other hand, if $\alpha=ce_n$, then $\alpha'=bg_{\theta_u}$ for some $u\in [m]$. Thus $\alpha'>\alpha$. Now suppose $v_i\neq v_1$. In this case, either $\alpha=af_l$ or $\alpha=bg_{\theta_{u_1}}$ or $\alpha=ce_n$, where $l\in [p]$, $u_1\in[m]$ and $n\in[s]$. Then by \Cref{lemma1} and considering the variables appearing in the monomials $a,b$ and $c$, one can verify that $\alpha'=af_{l'}$ or $\alpha'=bg_{\theta_{u_1'}}$ or $\alpha'=ce_{n'}$, respectively, where $l'\in [p]$, $u_1'\in[m]$ and $n'\in[s]$. Therefore, by the induction hypothesis, $\alpha'>\alpha$.  
		
		We now proceed to show that with respect to the above ordering, $\G(\hs_k(J(\Gp)))$ has linear quotients. It is easy to see that for each $2\le i\le p$, $\l af_1,\ldots,af_{i-1}\r:af_i=\l f_1,\ldots,f_{i-1}\r:f_i$, and thus by the induction hypothesis the colon ideal is generated by variables. 
		
		\noindent
		{\bf Claim 1}: For each $i\in[m]$, the colon ideal $\l af_1,\ldots,af_p,bg_{\theta_1},\ldots,bg_{\theta_{i-1}}\r:bg_{\theta_i}$ is generated by variables.
		
		\noindent
		\textit{Proof of Claim 1}: Let us first consider the monomial
        \[
h=\frac{\mathrm{lcm}(af_j,bg_{\theta_i})}{bg_{\theta_i}},
        \]
        for some $j\in[p]$ and $i\in[m]$. Let $g_{\theta_i}=\x_{\C}\x_{\sigma}$, where $\C$ is a minimal vertex cover of $\Gp\s w_{11}$, and $\sigma\subseteq\set_{\Gp\s w_{11}}(\C)$ with $|\sigma|=k$. One can see that $|\sigma\cap\N_{\Gp}[w_{11}]|\le 1$. Indeed, $w_{11}\notin\sigma$, and if $\beta,\gamma\in\sigma$ for some $\beta,\gamma\in\N_{\Gp}(w_{11})$, then, by \Cref{theorem 1}  $\beta,\gamma\notin\C$. However, since $w_{11}$ is a simplicial vertex, we have $\{\beta,\gamma\}\in E(\Gp)$, a contradiction to the fact that $\C$ is a vertex cover of $\Gp$.  

        Now, first consider the case when $w_{lq}\mid h$ for some $l\in[t]$ and $q\in[r_l]$. In this case, $w_{lq}\notin\{w_{11}\}\cup\C\cup\sigma$ and hence, by \Cref{proposition 4}, $v_l\in\C$. Thus $w_{lq}\nmid bg_{\theta_i}$ and $v_l\mid bg_{\theta_i}$. Therefore, by \Cref{lemma1} and the ordering described in \Cref{eq 1}, we have $w_{lq}\frac{bg_{\theta_i}}{v_l}=bg_{\theta_{i'}}$, where $bg_{\theta_{i'}}>bg_{\theta_i}$, by $(*)$. Observe that
		\[
		\frac{\lcm(bg_{\theta_{i'}},bg_{\theta_i})}{bg_{\theta_i}}=w_{lq},
		\]		
		where $w_{lq}\mid h$. Therefore, we may assume that $h\mid\prod_{i=1}^tv_i$. Our next aim is to show that $v_1\mid h$. Indeed, if $v_1\nmid h$, then $v_1\in\C$. Therefore, $w_{1l}\notin\C$ for some $l\in[r_1]$. Moreover, since $w_{1l}\nmid h$, we have $w_{1l}\in\sigma$, and hence $|\sigma\cap\N_{\Gp}[w_{11}]|=1$. On the other hand, since $h\mid\prod_{i=2}^tv_i$, we have that $\N_{\Gp}[w_{11}]\subseteq \{w_{11}\}\cup\C\cup\sigma$. Take $\sigma'=\sigma\setminus\{w_{1l}\}$ and $\C'=\C\setminus \N_{\Gp}[w_{11}]$. Then $\{w_{11}\}\sqcup\C\sqcup\sigma=\N_{\Gp}[w_{11}]\sqcup\C'\sqcup \sigma'$. Moreover, since $\hs_k(J(\Gp))$ is an equigenerated monomial ideal, using \Cref{proposition 2} we have that $\C'$ is a minimal vertex cover of $\Gp\setminus\N_{\Gp}[w_{ij}]$. Therefore, $bg_{\theta_i}\in \G(a\cdot\hs_{k-1}(J(\Gp\s\N_{\Gp}[w_{11}])))$, a contradiction due to our chosen ordering in \Cref{eq 1}. Thus, from now onward, we may assume that $v_1\mid h$. 
        
        Again, by our assumption $\N_{\Gp}[w_{11}]\s\{v_1\}\subseteq \{w_{11}\}\cup\C\cup\sigma$ and $v_1\notin\C$.  Now, by \Cref{proposition 2}, $\C\s\N_{\Gp}(w_{11})$ is a vertex cover of $\Gp\s\N_{\Gp}[w_{11}]$. Choose $\B\subseteq\C\s\N_{\Gp}(w_{11})$ such that $\B$ is a minimal vertex cover of $\Gp\s\N_{\Gp}[w_{11}]$. Also, choose $\tau\subset\sigma$ such that $|\tau|=k-1$, and $\tau\cap\N_{\Gp}[w_{11}]=\emptyset$. Then by \Cref{proposition 3}, $\tau\subseteq\set_{\Gp\s\N_{\Gp}[w_{11}]}(\B)$. Consequently, $\x_{\N_{\Gp}[w_{11}]}\x_{\B}\x_{\tau}=af_m$ for some $m\in [p]$. Since $\N_{\Gp}[w_{11}]\s\{v_1\}\subseteq (\{w_{11}\}\cup\C\cup\sigma)$ and $v_1\notin\C$, we have
		\[
		\frac{\lcm(af_{m},bg_{\theta_i})}{bg_{\theta_i}}=v_1,
		\]
		where $v_1\mid h$.

        Next, we consider the monomial
        \[
        h_1=\frac{\lcm(bg_{\theta_{j}},bg_{\theta_i})}{bg_{\theta_i}}=\frac{\lcm(g_{\theta_{j}},g_{\theta_i})}{g_{\theta_i}},
        \]
        for some $j<i$. Since the ideal $\l g_1,\ldots,g_q\r$ has linear quotients with respect to the ordering $g_1>\ldots>g_q$, there exists $\gamma\in [q]$ with $\gamma<\theta_{i}$ such that $\frac{\lcm(g_{\gamma},g_{\theta_{i}})}{g_{\theta_i}}$ is a variable that divides $h_1$. Due to our given ordering, observe that $bg_{\gamma}\in \{af_1,\ldots,af_{p},bg_{\theta_1},\ldots,bg_{\theta_{i-1}}\}$. Hence, the colon ideal $\l af_1,\ldots,af_p,bg_{\theta_1},\ldots,bg_{\theta_{i-1}}\r:bg_{\theta_i}$ is generated by variables. This completes the proof of Claim 1.
		
		\noindent
		{\bf Claim 2}: For each $i\in[s]$, the colon ideal $\l af_1,\ldots,af_p,bg_{1},\ldots,bg_{m},ce_1,\ldots,ce_{i-1}\r:ce_i$ is generated by variables.
		
		\noindent
		\textit{Proof of Claim 2}: Observe that $\l ce_1,\ldots,ce_{i-1}\r:ce_i=\l e_1,\ldots,e_{i-1}\r:e_i$, and hence by the induction hypothesis, this colon ideal is generated by variables. Next, suppose
		\[
		h'=\frac{\lcm(\Delta,ce_i)}{ce_i},
		\]
		where $\Delta\in\{af_1,\ldots,af_p,bg_{1},\ldots,bg_{m}\}$. Then it is easy to see that $w_{11}\mid h'$. Observe that $\alpha=ce_i\in\hs_k(J(\Gp))$, with $w_{11}\nmid \alpha$ and $v_1\mid \alpha$. Hence, by \Cref{lemma1} and the ordering in \Cref{eq 1}, $\alpha'=w_{11}\frac{ce_i}{v_1}\in \hs_k(J(\Gp))$ and $\alpha'>\alpha$, where 
		\[
		\frac{\lcm(\alpha',ce_i)}{ce_i}=w_{11}.
		\]
		This completes the proof of Claim 2 and subsequently, the proof of the theorem.
		\end{proof}

A generalization of polymatroidal ideals is known as weakly polymatroidal ideals, introduced by Kokubo-Hibi \cite{KokubuHibiWeakPoly}. The weakly polymatroidal property is stronger than the linear quotient property; in particular, weakly polymatroidal ideals always have linear quotients \cite[Theorem 1.3]{MohammadiMordaliWeakPoly}. For Cohen-Macaulay Cameron-Walker graphs and certain clique-corona graphs, we will show that the homological shift ideals of their vertex cover ideals are, in fact, weakly polymatroidal. Let us begin by recalling the definition.
        
\begin{definition}[weakly polymatroidal ideal]
		A monomial ideal $I\subset R=\K[x_1,\ldots,x_n]$ is said to be a weakly
		polymatroidal ideal if for any two minimal monomial generators $u=x_1^{\alpha_1}\cdots x_n^{\alpha_n}>_{\lex}v=x_1^{\beta_1}\cdots x_n^{\beta_n}$ with $\alpha_1=\beta_1,\ldots,\alpha_{t-1}=\beta_{t-1}$ and $\alpha_t>\beta_t$ for some $t$, there exists some $j > t$ such that $x_t(v/x_j ) \in\G(I)$. Here, the lexicographic order on the monomials in $R$ is induced from the order of the variables $x_1>x_2>\dots>x_n$.
	\end{definition}
	

    The Cohen-Macaulay Cameron-Walker graphs are combinatorially classified by Hibi-Higashitani-Kimura-O'Keefe in \cite{HibiHigashitaniKimuraCamWalk} as follows.
    \begin{theorem}[{\cite[Theorem 1.3]{HibiHigashitaniKimuraCamWalk}}]\label{thm:cmcw_structure}
        Let $G$ be a Cameron-Walker graph. Then the following are equivalent:
        \begin{enumerate}
            \item $G$ is Cohen-Macaulay;
            \item Each connected component of $G$ is either a $K_2$ or a $K_3$ or consists of a connected bipartite graph with vertex partition $\{1,\ldots,m\}\sqcup \{m+1,\ldots,n\}$ such that there is exactly one pendant triangle attached to each vertex $i\in \{1\ldots,m\}$ and that there is exactly one leaf edge attached to each vertex $j\in \{m+1,\ldots, n\}$.
        \end{enumerate}
            
        \end{theorem}

    \begin{theorem}\label{Cameron-Walker graph}
        Let $G$ be a Cohen-Macaulay Cameron-Walker graph. Then $\hs_k(J(G))$ is weakly polymatroidal, and thus, has linear quotients for all $k\geq 0$.
    \end{theorem}

    \begin{proof}
        From the structure of Cohen-Macaulay Cameron-Walker graphs given in \Cref{thm:cmcw_structure}, one can observe that those actually belong to the class of clique-whiskered graphs. If $G$ consists only of $K_2$ or $K_3$ as connected components, then $G$ is a chordal graph, and thus, the result holds by \Cref{theorem 1}. Therefore, without loss of generality, we can assume $G$ is not a disjoint union of several $K_2$ or $K_3$. To keep the notation consistent as used earlier for clique-whiskered graphs, we use the following labeling of the vertices of $G$:
        \begin{align*}
            V_G&=\{w_{i1},w_{i2},v_{i}\mid 1\leq i\leq m\} \sqcup \{w_{j1},v_{j}\mid m+1\leq j\leq n\},\\
            E_G&=E_H\cup \{\{w_{i1},w_{i2}\},\{v_i,w_{i1}\},\{v_i,w_{i2}\}\mid 1\leq i\leq m\}\cup \{\{v_j,w_{j1}\}\mid m+1\leq j\leq n\},
        \end{align*}
        where $H$ is a connected bipartite graph with the vertex partition $\{w_{11},\ldots, w_{m,1}\}\sqcup \{w_{(m+1)1},\ldots,w_{n,1}\}$.
We will show that $\hs_k(J(G))$ is weakly polymatroidal with respect to the lexicographic order induced from the following ordering of the vertices:
\[ w_{11}>w_{12}>v_1>\cdots >w_{m1}>w_{m2}>v_m>w_{(m+1)1}>v_{m+1}>\cdots >w_{n1}>v_n. \]
Let $f,g\in \G(\hs_k(J(G)))$ with $f>_{\lex} g$. Since $J(G)$ has the linear quotient property, by \Cref{remark1}, we can write $f=\x_{\C}\x_{\sigma}$ and $g=\x_{\C'}\x_{\sigma'}$ for some minimal vertex covers $\C$ and $\C'$ of $G$, where $\sigma\subseteq \set_{G}(\C), \sigma'\subseteq \set_{G}(\C')$ with $\vert \sigma\vert=\vert \sigma'\vert=k$. Now, let us consider 
\[z=\max\{a\in V_G\mid a\in (\C\cup \sigma)\setminus (\C'\cup \sigma')\}.\]
Since $f>_{\lex} g$, we have 
\begin{equation}\label{eq:star}\tag{$\star$}
    b\in \C\cup \sigma \iff b\in \C'\cup\sigma' \text{ for each } b>z.
\end{equation}
\noindent \textbf{Case-I:} Let $z=w_{pq}$ for some $p$ and $q$. Then, $w_{pq}\not\in \C'\cup \sigma'$ implies $v_p\in \C'$. In particular, $g\in \G(\hs_k(J(G)))$ such that $v_p\mid g$ and $w_{pq}\nmid g$. Thus, by \Cref{lemma1}, $z(g/v_p)\in \hs_{k}(J(G))$ with $v_p<z=w_{pq}$.

\noindent \textbf{Case-II:} Let $z=v_i$ for some $i\in [n]$. \\
\noindent \textbf{Claim:} There exists $j>i$ such that $w_{jl}\in \sigma'$ for some $l$.\\
\textit{Proof of the claim.} Since $v_i\not\in \C'$, we have $w_{is}\in \C'$ for all possible $s$. Thus, $w_{is}\not\in \sigma'$ for all $s$. Since $w_{is}>v_i$, it follows from \Cref{eq:star} that $w_{is}\in \C\cup \sigma$ for all $s$. Now, $v_{i}\in \C$ implies $w_{it}\not\in \C$ for some $t$, and thus, $w_{it}\in \sigma\setminus \sigma'$ since $w_{it}\in\C'$. Now, we proceed to show that for each $j<i$, if $w_{jl}\in \sigma'$ for some $l$, then $w_{jp}\in \sigma$ for some $p$. Indeed, $w_{jl}\in \sigma'$ implies $v_j\in\C'$. Hence, $v_j\in \C$ by \Cref{eq:star} as $v_j>v_i$. Therefore, $w_{jp}\not\in \C$ for some $p$. Observe that $w_{jl}\in\sigma'$ implies $w_{jp}\in \C'\cup\sigma'$. Thus, again due to \Cref{eq:star}, $w_{jp}\in \sigma$ since $w_{jp}>v_i$, as desired. Now, $\vert \sigma\vert=\vert \sigma'\vert=k$, $w_{it}\in \sigma\setminus \sigma'$, and our previous arguments together imply there exists $j>i$ such that $w_{jl}\in\sigma'\setminus \sigma$ for some $l$. In this situation, let us consider two subcases.\\
\noindent\textbf{Subcase-I:} Suppose $z=v_i$ for some $i\in [m]$. In this case, $w_{i1},w_{i2}\in\C'\setminus \sigma'$. Take $\C''=(\C'\setminus w_{i2})\cup v_{i}$ and $\sigma''=(\sigma'\setminus w_{jl})\cup \{w_{i2}\}$. Then, it is easy to see that $\C''$ is a minimal vertex cover of $G$ and by \Cref{theorem 1}, $\set_{G}(\C'')=\set_{G}(\C')\cup \{w_{i2}\}$. Again, we have 
\[ \sigma''=(\sigma'\setminus w_{jl})\cup \{w_{i2}\}\subseteq (\set_{G}(\C')\setminus \{w_{jl}\})\cup \{w_{i2}\}=\set_{G}(\C'')\setminus \{w_{jl}\}\subsetneq\set_{G}(\C'').\]
Therefore, $\x_{\C''}\x_{\sigma''}=z(g/w_{jl})\in\hs_{k}(J(G))$ with $w_{jl}<z=v_i$.\\
\noindent\textbf{Subcase-II:} Let us assume $z=v_i$ for some $i\in\{m+1,\ldots,n\}$. Then $v_i\not\in \C'$ and $w_{i1}\in\C'\setminus \sigma'$. Let $N_{G}(w_{i1})\setminus \{v_1\}=\{w_{q_{1}1},\ldots,w_{q_{s}1}\}$. Since the graph $H$ is bipartite, from our labeling it follows that $\{q_1,\ldots,q_s\}\subseteq \{1,\ldots,m\}$. Without loss of generality, suppose there exists some $r\leq s$ such that 
\[ \C'\cap \{w_{q_{1}1},\ldots,w_{q_{s}1}\}=\{w_{q_{r+1}1},\ldots,w_{q_{s}1}\}. \]
Then, $w_{q_{1}1},\ldots,w_{q_{r}1}\not\in\C'$. Now, $v_{i}\in \C\setminus \C'$ and $w_{i1}>v_i$ together imply $w_{i1}\in (\sigma \setminus \sigma')\cap (\C'\setminus \C)$ by \Cref{eq:star} and by \Cref{proposition 4}. Note that for each $t\in [r]$, $w_{q_{t}1}\not\in \C'$ implies $v_t,w_{q_{t}2}\in \C'$, and thus, $w_{q_{t}2}\not\in \sigma'$. Now, $v_i\in \C$ gives $w_{i1}\not\in \C$, which further implies $\{w_{q_{1}1},\ldots,w_{q_{s}1}\}\subseteq \C$. Therefore, we have $w_{q_{t}1}\in \C\setminus \C'$ for each $t\in [r]$. Since $w_{q_{t}1}>v_i$, we have $w_{q_{t}1}\in \sigma'$ by \Cref{eq:star} for each $t\in [r]$. In this situation, we consider
\begin{align*}
    \C_{1}'&=(\C'\setminus \{w_{q_{1}2},\ldots, w_{q_{r}2}\})\cup \{w_{q_{1}1},\ldots, w_{q_{r}1}\},\\
\sigma_{1}'&=(\sigma' \setminus \{w_{q_{1}1},\ldots, w_{q_{r}1}\})\cup \{w_{q_{1}2},\ldots, w_{q_{r}2}\}.
\end{align*}
Then one can observe that $\C_{1}'\cup \sigma_{1}'=\C'\cup \sigma'$, $\C_{1}'$ is a minimal vertex cover of $G$, and $\sigma_{1}'\subseteq \set_{G}(\C_{1}')$. Thus, we can view $\x_{\C'}\x_{\sigma'}\in\hs_{k}(J(G))$ as $\x_{\C_{1}'}\x_{\sigma_{1}'}\in\hs_{k}(J(G))$. Now, take
\[ \C_1''=(\C_{1}'\setminus \{w_{i1}\})\cup \{v_{i}\} \text{ and } \sigma_1''=(\sigma_{1}'\setminus \{w_{jl}\})\cup \{w_{i1}\}. \]
It follows from our choice of $\C_1'$ that $N_{G}(w_{i1})\subseteq \C_1''$, and thus, $\C_1''$ is a minimal vertex cover of $G$. Moreover, we have $\sigma_1''\subseteq \set_{G}(\C_1'')$ by \Cref{theorem 1}, and $\vert \sigma_1''\vert=\vert \sigma_{1}'\vert=\vert \sigma'\vert=k$. Hence, $\x_{\C_1''}\x_{\sigma_1''}=v_{i}(g/w_{jl})=z(g/w_{jl})\in\hs_{k}(J(G))$ with $w_{jl}<z=v_i$. This completes the proof.
    \end{proof}

    \begin{definition}[clique corona graphs]\label{def:clique-corona}
		Let $\Gamma$ be a graph on the vertex set $\{w_{i1}\mid i\in[n]\}$. Let $\H=\{K_{t_1},\ldots,K_{t_n}\}$ be a collection of complete graphs. Then the corona product $\Gamma\circ \H$ is a finite simple graph $G$, called a clique corona graph, with the vertex set and edge set as follows. 
		\begin{align*}
			V_G&=V_{\Gamma}\cup(\cup_{i=1}^n V_{K_{t_i}})\\
			E_G&=E_{\Gamma}\cup(\cup_{i=1}^n E_{K_{t_i}})\cup(\cup_{i=1}^n A_i),
		\end{align*}
		 where $A_i=\{\{w_{i1},x\}\mid x\in V_{K_{t_i}}\}$ for each $i\in[n]$. Thus, $G$ is obtained by taking disjoint union of $\Gamma$ and $K_{t_i}$, with additional edges joining each vertex
		 $w_{i1}$ to all vertices $K_{t_i}$. 
	\end{definition}

    \begin{remark}
       From \Cref{def:clique-corona} it is clear that a clique corona graph is also a clique-whiskered graph. Moreover, when each $t_i=1$, the clique corona graph is nothing but a whiskered graph. Thus, due to \Cref{theorem 200}, we see that all the homological shifts ideals of the vertex cover ideal of a clique-corona graph may not have linear quotients. However, if we take each $t_i\geq 2$ in \Cref{def:clique-corona}, then we will show in the next theorem that the linear quotient property holds true. In fact, we show that the corresponding ideals are weakly polymatroidal.
    \end{remark}

    \begin{theorem}\label{clique corona}
        Let $G=\Gamma \circ \H$ be a clique corona graph with $\H=\{K_{t_1},\ldots,K_{t_n}\}$. If $t_i\geq 2$ for each $i\in [n]$, then $\hs_k(J(G))$ is weakly polymatroidal, and thus, linear quotients for all $k\geq 0$.
    \end{theorem}
    \begin{proof}
        Let $V_{\Gamma}=\{w_{11},\ldots,w_{n1}\}$ and $V_{K_{t_i}}=\{w_{i2},\ldots,w_{it_{i}},v_{i}\}$ for each $i\in [n]$. Consider the lexicographic ordering induced by the following:
        \[ w_{11}>w_{12}>\cdots>w_{1t_1}>v_1>\cdots> w_{n1}>w_{n2}>\cdots>w_{nt_n}>v_n. \]
        Then, repeating the same argument up to the Subcase-I of the proof of \Cref{Cameron-Walker graph}, one can see that $\hs_k(J(G))$ is weakly polymatroidal for each $k\geq 0$. Consequently, $\hs_k(J(G))$ has linear quotients for all $k\ge 0$. 
    \end{proof}

    \subsection*{Final Remark} Based on experimental evidence, it was conjectured by Bandari, Bayati, and Herzog that the homological shift ideals of a polymatroidal ideal are again polymatroidal. This conjecture has been proven in several cases: for polymatroidal ideals with strong exchange property \cite{HMRZ2021}, for square-free polymatroidal ideals in \cite{BayatiMatroidal}, for the first homological shift ideal of any polymatroidal ideal in \cite{BandariPolyLinRes, FicarraHomShiftPoly}, and others. However, it remains open in full generality. Naturally, one can ask the same question for weakly polymatroidal ideals, as they generalize polymatroidal ideals. Note that the vertex cover ideals of the graphs considered in \Cref{theorem 200} are square-free weakly polymatroidal ideals (see \cite[Theorem 4.3]{LuWangPowVerCovIdeal}), and from our \Cref{theorem 200} it follows that their homological shift ideals may not be weakly polymatroidal.

    \noindent
{\bf Acknowledgements.} The authors would like to thank Abhay Soman for some helpful comments. Roy is supported by a Postdoctoral Fellowship at the Chennai Mathematical Institute. Saha thanks the National Board for Higher Mathematics (India) for the financial support through the NBHM Postdoctoral Fellowship. Both authors are partially supported by a grant from the Infosys Foundation. 

    \bibliographystyle{abbrv}
\bibliography{ref}
	\end{document}